\documentclass[11pt]{amsart}
\usepackage[T1]{fontenc}
\usepackage{lmodern}
\usepackage[utf8]{inputenc}
\usepackage{amsmath,amssymb,amsfonts,amsthm}
\usepackage{cite}
\usepackage{tikz}
\usepackage{geometry}
\usepackage[hidelinks]{hyperref}
\input{glyphtounicode}
\hypersetup{
  pdftitle={Cesàro Means Along Polynomial Subsequences of Fourier Partial Sums at Lebesgue Points},
  pdfauthor={Ushangi Goginava},
  pdfsubject={Fourier series and Cesàro summability},
  pdfkeywords={Fourier series, Cesàro means, polynomial subsequences,
    Lebesgue points, Weyl sums, Zalcwasser problem}
}

\numberwithin{equation}{section}
\theoremstyle{plain}
\newtheorem{theorem}{Theorem}
\newtheorem{lemma}[theorem]{Lemma}

\theoremstyle{definition}

\theoremstyle{remark}

\newcommand{\T}{\mathbb{T}}

\begin{document}

\title[Ces\`aro means along polynomial subsequences]{Ces\`aro Means along
Polynomial Subsequences of Fourier Partial Sums at Lebesgue Points}
\author{Ushangi Goginava}
\address{U. Goginava, Department of Mathematical Sciences\\
United Arab Emirates University, P.O. Box No. 15551\\
Al Ain, Abu Dhabi, UAE}
\email{zazagoginava@gmail.com; ugoginava@uaeu.ac.ae}
\date{}

\begin{abstract}
In 1936, Zalcwasser proved the almost everywhere convergence of the
arithmetic means of the square subsequence of trigonometric Fourier partial
sums and asked whether this result extends to higher powers and to Ces\`aro
means of fractional order. We give affirmative answers to both questions in
a stronger pointwise form. Let $0<\alpha\leq 1$, and let $P$ be an
integer-valued polynomial of degree with positive leading
coefficient. We prove that the $(C,\alpha)$ means of the Fourier partial sums
along any sequence whose $k$-th term equals $P(k)$ for all sufficiently large
$k$ converges to $f(x)$ at every Lebesgue point $x$ of every
$f\in L^{1}(\T)$.
\end{abstract}

\maketitle

\section{Introduction}

Let $\T=\mathbb{R}/(2\pi\mathbb{Z})$, and let $f\in L^{1}(\T)$. The Fourier
coefficients of $f$ are defined by
\begin{equation*}
\widehat{f}(m):=\frac{1}{2\pi}\int_{-\pi}^{\pi}f(x)e^{-imx}\,dx,
\qquad m\in\mathbb{Z}.
\end{equation*}

The $n$-th partial sum of the Fourier series of $f$ is defined by
\begin{equation*}
S_nf(x):=\sum_{|m|\leq n}\widehat{f}(m)e^{imx},
\qquad n=0,1,2,\ldots.
\end{equation*}

For $\beta>-1$, put
\begin{equation*}
A_n^\beta:=\binom{n+\beta}{n},
\qquad n=0,1,2,\ldots.
\end{equation*}
For $0<\alpha\leq 1$ and a strictly increasing sequence
$a=(a_k)_{k\geq 1}$ of positive integers, define the $(C,\alpha)$ means by
\begin{equation*}
\sigma_{N,\alpha}^{a}f(x)
:=\frac{1}{A_{N-1}^{\alpha}}
\sum_{k=1}^{N}A_{N-k}^{\alpha-1}S_{a_k}f(x),
\qquad N\geq 1.
\end{equation*}
For $\alpha=1$, these are the arithmetic means.

We say that $x\in\T$ is a Lebesgue point of $f\in L^{1}(\T)$ if
\begin{equation*}
\int_{-h}^{h}|f(x+t)-f(x)|\,dt=o(h)
\qquad (h\downarrow 0).
\end{equation*}

It is a classical theorem of Lebesgue that the $(C,\alpha)$ means, with
$\alpha>0$, of the partial sums of the Fourier series of an integrable
function converge to the function at every Lebesgue point; see
Lebesgue~\cite{Lebesgue1905} or Zygmund~\cite[Ch.~III]{Zygmund}.

In 1936, Zalcwasser~\cite{Zalcwasser1936} investigated the arithmetic
means of partial sums of Fourier series of functions $f\in L^{1}(\T)$ along
the subsequence $n_k=k^2$. He proved that
\begin{equation}
\frac{1}{m+1}\sum_{k=0}^{m}S_{k^2}f(x)\longrightarrow f(x)
\label{k2}
\end{equation}
almost everywhere for every $f\in L^{1}(\T)$, and posed two natural
problems:

\begin{enumerate}
\item Can the subsequence in \eqref{k2} be replaced by the more general
subsequence $n_k=k^r$, where $r\in\mathbb{N}$ and $r>2$?

\item Can the arithmetic means be replaced by the more general $(C,\alpha)$
means with $0<\alpha<1$?
\end{enumerate}

Convergence at Lebesgue points for arithmetic means was studied by
Belinsky and by Izumi and Kawata
\cite{Bel1,Bel2,IzumiKawata1940}. We also mention the results of G\'at
\cite{gat1,gat2} concerning almost-everywhere convergence and divergence of
Ces\`aro means of subsequences of Fourier partial sums. The above-mentioned
results do not cover the two questions of Zalcwasser. In this paper, we give
positive answers to both questions. More precisely, we prove the following
theorem.

\begin{theorem}
\label{main}
Let $0<\alpha\leq 1$. Let $P\in\mathbb{Q}[x]$ be an integer-valued
polynomial of degree $r\geq 2$ with positive leading coefficient. Let
$a=(a_k)$ be a strictly increasing sequence of positive integers such that
\begin{equation*}
a_k=P(k)
\end{equation*}
for all sufficiently large $k$. Then, for every $f\in L^{1}(\T)$ and every
Lebesgue point $x$ of $f$,
\begin{equation*}
\sigma_{N,\alpha}^{a}f(x)\longrightarrow f(x).
\end{equation*}
\end{theorem}

Taking $P(k)=k^r$,  gives affirmative
answers to both questions of Zalcwasser.

\section{Proof of the Main Result}

Before proving the main theorem, we give some preliminary results. The
Ces\`aro numbers satisfy
\begin{equation}
\sum_{j=0}^{N}A_j^{\alpha-1}=A_N^\alpha,
\qquad N\geq 0,
\label{cesaro-identity}
\end{equation}
and, for every $\beta>-1$,
\begin{equation}
A_n^\beta\asymp_\beta (n+1)^\beta,
\qquad n\geq 0.
\label{cesaro-asymptotic}
\end{equation}

Set
\begin{equation*}
\varphi_x(t):=f(x+t)+f(x-t)-2f(x),
\qquad 0<t\leq\pi,
\end{equation*}
and
\begin{equation*}
\Phi_x(h):=\int_0^h|\varphi_x(t)|\,dt.
\end{equation*}
At every Lebesgue point $x$ of $f$, we have
$\Phi_x(h)=o(h)$ as $h\downarrow 0$. The usual Dirichlet-kernel
representation gives
\begin{equation}
\sigma_{N,\alpha}^{a}f(x)-f(x)
=\frac{1}{\pi}\int_0^\pi\varphi_x(t)K_{N,\alpha}(t)\,dt,
\label{kernel-representation}
\end{equation}
where
\begin{equation*}
K_{N,\alpha}(t)
:=\frac{1}{A_{N-1}^{\alpha}}
\frac{H_{N,\alpha}(t)}{2\sin(t/2)},
\end{equation*}
and
\begin{equation*}
H_{N,\alpha}(t)
:=\sum_{k=1}^{N}A_{N-k}^{\alpha-1}
\sin\bigl((a_k+1/2)t\bigr).
\end{equation*}

For the proof of the main theorem, we need two auxiliary lemmas.

\begin{lemma}[Polynomial Weyl estimate]
\label{Weyl}
Let $r\geq 2$ and $c\in\mathbb{R}\setminus\{0\}$. There exists
$\eta_0=\eta_0(r)>0$ such that, for every $0<\eta<\eta_0$, there is a
constant $C=C(c,r,\eta)$ with the following property. If $Q$ is a real
polynomial of degree $r$ whose leading coefficient is $ct$, then, for every
integer $L\geq 1$ and every $0<t\leq 1$,
\begin{equation}
\left|\sum_{j=0}^{L-1}e^{iQ(j)}\right|
\leq C\left(Lt^\eta+L^{1-\eta}
+t^{-\eta}L^{1-r\eta}\right).
\label{weyl-estimate}
\end{equation}
The constant is independent of the lower-order coefficients of $Q$.
\end{lemma}

\begin{proof}
By the van der Corput $r$-th derivative estimate for exponential sums
(see, for example, \cite{GrKol}), there exist constants $\rho_r,\tau_r>0$,
depending only on $r$, such that, 
\begin{equation*}
\frac{1}{L}\left|\sum_{j=0}^{L-1}e^{iQ(j)}\right|
\leq C\left(t^{\rho_r}+L^{-\tau_r}
+(tL^r)^{-\tau_r}\right)
\end{equation*}
whenever $t$ is sufficiently small. Values of $t$ bounded away from zero
are covered by the trivial estimate. Take
$\eta_0:=\min\{\rho_r,\tau_r\}$. If $tL^r\geq 1$, then, for
$0<\eta<\eta_0$, each term on the right is bounded by the corresponding
term with exponent $\eta$. If $tL^r<1$, the trivial estimate is bounded by
$t^{-\eta}L^{1-r\eta}$. This proves \eqref{weyl-estimate}. Since the
$r$-th derivative of $Q$ depends only on its leading coefficient, the
constant is uniform in all lower-order coefficients.
\end{proof}

\begin{lemma}
\label{mainlemma}
Let $0<\alpha\leq 1$, and let $a=(a_k)$ be a strictly increasing sequence
of positive integers. Suppose that there are $\delta_0\in(0,\pi]$ and
nonnegative, nonincreasing, locally absolutely continuous functions $M_N$
on $(0,\delta_0]$ such that
\begin{equation}
|K_{N,\alpha}(t)|\leq M_N(t),
\qquad a_N^{-1}\leq t\leq\delta_0,
\label{l1}
\end{equation}
and
\begin{equation}
\sup_N\left\{a_N^{-1}M_N(a_N^{-1})
+\int_{a_N^{-1}}^{\delta_0}M_N(t)\,dt\right\}<\infty.
\label{l2}
\end{equation}
Then, for every $f\in L^{1}(\T)$ and every Lebesgue point $x$ of $f$,
\begin{equation*}
\sigma_{N,\alpha}^{a}f(x)\longrightarrow f(x).
\end{equation*}
\end{lemma}

\begin{proof}
Fix $f\in L^{1}(\T)$ and a Lebesgue point $x$ of $f$. Given
$\varepsilon>0$, choose $\delta\in(0,\delta_0]$ such that
\begin{equation}
\Phi_x(h)\leq\varepsilon h,
\qquad 0<h\leq\delta.
\label{phi-small}
\end{equation}
For all sufficiently large $N$, we have $a_N^{-1}<\delta$, and we split
\begin{align}
\int_0^\pi\varphi_x(t)K_{N,\alpha}(t)\,dt
={}&\int_0^{a_N^{-1}}\varphi_x(t)K_{N,\alpha}(t)\,dt
+\int_{a_N^{-1}}^\delta\varphi_x(t)K_{N,\alpha}(t)\,dt \notag\\
&+\int_\delta^\pi\varphi_x(t)K_{N,\alpha}(t)\,dt.
\label{1-3}
\end{align}

By \eqref{cesaro-identity}, the coefficients in the definition of
$K_{N,\alpha}$ are nonnegative and sum to one. Moreover,
\begin{equation*}
\left|\frac{\sin((a_k+1/2)t)}{2\sin(t/2)}\right|
\leq a_k+\frac12\leq a_N+\frac12, k\leq N.
\end{equation*}
Consequently, $|K_{N,\alpha}(t)|\leq Ca_N$ for all $t$, and therefore
\begin{equation}
\left|\int_0^{a_N^{-1}}\varphi_x(t)K_{N,\alpha}(t)\,dt\right|
\leq Ca_N\Phi_x(a_N^{-1})
\leq C\varepsilon.
\label{small-part}
\end{equation}

For the middle integral, \eqref{l1}, integration by parts, and
\eqref{phi-small} give
\begin{align}
\left|\int_{a_N^{-1}}^\delta
\varphi_x(t)K_{N,\alpha}(t)\,dt\right|
&\leq\int_{a_N^{-1}}^\delta M_N(t)\,d\Phi_x(t) \notag\\
&=M_N(\delta)\Phi_x(\delta)
-M_N(a_N^{-1})\Phi_x(a_N^{-1})
-\int_{a_N^{-1}}^\delta\Phi_x(t)M_N'(t)\,dt \notag\\
&\leq\varepsilon\delta M_N(\delta)
+\varepsilon\int_{a_N^{-1}}^\delta t\bigl(-M_N'(t)\bigr)\,dt \notag\\
&=\varepsilon\left(
a_N^{-1}M_N(a_N^{-1})
+\int_{a_N^{-1}}^\delta M_N(t)\,dt\right)
\leq C\varepsilon.
\label{middle-part}
\end{align}

It remains to treat the last integral. Write
\begin{equation}
\int_\delta^\pi\varphi_x(t)K_{N,\alpha}(t)\,dt
=\frac{1}{A_{N-1}^\alpha}
\sum_{k=1}^N A_{N-k}^{\alpha-1}c_k,
\label{tail-sum}
\end{equation}
where
\begin{equation*}
c_k:=\int_\delta^\pi
\varphi_x(t)\frac{\sin((a_k+1/2)t)}{2\sin(t/2)}\,dt.
\end{equation*}
Since $\varphi_x(t)/(2\sin(t/2))\in L^1([\delta,\pi])$, the
Riemann--Lebesgue lemma implies that $c_k\to 0$ as $k\to\infty$.
Furthermore, the numbers
\begin{equation*}
\lambda_{N,k}:=
\frac{A_{N-k}^{\alpha-1}}{A_{N-1}^\alpha},
\qquad 1\leq k\leq N,
\end{equation*}
are nonnegative and satisfy $\sum_{k=1}^N\lambda_{N,k}=1$. For every fixed
$k$, \eqref{cesaro-asymptotic} gives
\begin{equation*}
\lambda_{N,k}
=\frac{A_{N-k}^{\alpha-1}}{A_{N-1}^\alpha}
\asymp\frac{1}{N}\longrightarrow 0.
\end{equation*}
Hence the Toeplitz lemma yields
\begin{equation}
\int_\delta^\pi\varphi_x(t)K_{N,\alpha}(t)\,dt\longrightarrow 0.
\label{tail-part}
\end{equation}
Combining \eqref{1-3}, \eqref{small-part}, \eqref{middle-part}, and
\eqref{tail-part}, and then letting $\varepsilon\downarrow 0$, completes the
proof.
\end{proof}

\begin{proof}[Proof of Theorem~\ref{main}]
Let $r$ be the degree of $P$, and write $c_r\neq 0$ for its leading
coefficient. Choose
\begin{equation}
0<\theta<\min\{\eta_0(r),\alpha/r\},
\label{theta-choice}
\end{equation}
where $\eta_0(r)$ is given by Lemma~\ref{Weyl}. We first establish the
weighted exponential-sum estimate
\begin{equation}
\left|\frac{1}{A_{N-1}^\alpha}
\sum_{k=1}^N A_{N-k}^{\alpha-1}e^{iP(k)t}\right|
\leq C\left(t^\theta+N^{-\theta}
+N^{-r\theta}t^{-\theta}\right)
\label{cesest}
\end{equation}
for all $N\geq 1$ and $0<t\leq 1$.

Indeed, changing variables $j=N-k$ gives
\begin{equation*}
\sum_{k=1}^N A_{N-k}^{\alpha-1}e^{iP(k)t}
=\sum_{j=0}^{N-1}A_j^{\alpha-1}e^{iP(N-j)t}.
\end{equation*}
For $0\leq M\leq N-1$, Lemma~\ref{Weyl}, applied to the polynomial
$u\mapsto P(N-u)t$, yields the following estimate. Here $L=M+1$, and the
leading coefficient is $(-1)^r c_r t$.
\begin{equation}
\left|\sum_{j=0}^{M}e^{iP(N-j)t}\right|
\leq C\left((M+1)t^\theta+(M+1)^{1-\theta}
+t^{-\theta}(M+1)^{1-r\theta}\right).
\label{partial-weyl}
\end{equation}
Abel summation gives
\begin{align*}
\sum_{j=0}^{N-1}A_j^{\alpha-1}e^{iP(N-j)t}
={}&\sum_{M=0}^{N-2}
\bigl(A_M^{\alpha-1}-A_{M+1}^{\alpha-1}\bigr)
\sum_{j=0}^{M}e^{iP(N-j)t}\\
&+A_{N-1}^{\alpha-1}
\sum_{j=0}^{N-1}e^{iP(N-j)t}.
\end{align*}
For $0<\alpha<1$,
\begin{equation*}
A_M^{\alpha-1}-A_{M+1}^{\alpha-1}
=\frac{1-\alpha}{M+1}A_M^{\alpha-1}
\leq C(M+1)^{\alpha-2},
\end{equation*}
while for $\alpha=1$ the difference is zero. Consequently, if
$\beta\in\{1,1-\theta,1-r\theta\}$, then \eqref{theta-choice} implies
$\alpha+\beta>1$, and
\begin{equation}
\sum_{M=0}^{N-2}
\bigl(A_M^{\alpha-1}-A_{M+1}^{\alpha-1}\bigr)(M+1)^\beta
+A_{N-1}^{\alpha-1}N^\beta
\leq C N^{\alpha+\beta-1}.
\label{weighted-power-sum}
\end{equation}
Indeed,
\begin{equation*}
\sum_{M=0}^{N-2}(M+1)^{\alpha+\beta-2}
\leq C N^{\alpha+\beta-1}
\end{equation*}
because $\alpha+\beta>1$, while
$A_{N-1}^{\alpha-1}N^\beta\leq C N^{\alpha+\beta-1}$.
Using \eqref{partial-weyl} and \eqref{weighted-power-sum}, we obtain
\begin{equation*}
\left|\sum_{j=0}^{N-1}A_j^{\alpha-1}e^{iP(N-j)t}\right|
\leq C\left(N^\alpha t^\theta+N^{\alpha-\theta}
+t^{-\theta}N^{\alpha-r\theta}\right).
\end{equation*}
Division by $A_{N-1}^\alpha\asymp N^\alpha$ proves \eqref{cesest}.

Now choose $K_0$ such that $a_k=P(k)$ for all $k\geq K_0$. The finitely
many initial discrepancies contribute
\begin{align*}
&\left|\frac{1}{A_{N-1}^\alpha}
\sum_{k=1}^{K_0-1}A_{N-k}^{\alpha-1}
\bigl(e^{ia_kt}-e^{iP(k)t}\bigr)\right|\\
&\hspace{4cm}\leq \frac{C}{N}\leq C N^{-\theta}.
\end{align*}
After increasing the constant, the estimate also covers the finitely many
remaining values of $N$. 

Since
\begin{equation*}
H_{N,\alpha}(t)
=\operatorname{Im}\left(
e^{it/2}\sum_{k=1}^N A_{N-k}^{\alpha-1}e^{ia_kt}
\right)
\end{equation*}
and $\sin(t/2)\geq c_0t$ for $0<t\leq 1$, we obtain
\begin{equation}
|K_{N,\alpha}(t)|
\leq C\left(t^{\theta-1}+N^{-\theta}t^{-1}
+N^{-r\theta}t^{-\theta-1}\right),
\qquad 0<t\leq 1.
\label{kernel-bound}
\end{equation}

Set $\delta_0:=1$ and
\begin{equation}
M_N(t):=C\left(t^{\theta-1}+N^{-\theta}t^{-1}
+N^{-r\theta}t^{-\theta-1}\right),
\qquad 0<t\leq 1.
\label{majorant}
\end{equation}
The function $M_N$ is nonnegative, nonincreasing, and locally absolutely
continuous. Since $a_N=P(N)$ for all sufficiently large $N$ and $P$ has
degree $r$ with positive leading coefficient,
\begin{equation*}
a_N\asymp N^r.
\end{equation*}
Thus
\begin{align*}
a_N^{-1}M_N(a_N^{-1})
&\leq C\left(a_N^{-\theta}+N^{-\theta}
+N^{-r\theta}a_N^\theta\right)
\leq C,
\end{align*}
and
\begin{align*}
\int_{a_N^{-1}}^1M_N(t)\,dt
&\leq C\left(
\int_{a_N^{-1}}^1t^{\theta-1}\,dt
+N^{-\theta}\int_{a_N^{-1}}^1\frac{dt}{t}
+N^{-r\theta}\int_{a_N^{-1}}^1t^{-\theta-1}\,dt
\right)\\
&\leq C\left(1+N^{-\theta}\log a_N
+N^{-r\theta}a_N^\theta\right)
\leq C.
\end{align*}
After increasing the constant once more to cover finitely many initial
indices, $M_N$ satisfies \eqref{l1} and \eqref{l2} for every $N$.
Lemma~\ref{mainlemma} therefore completes the proof.
\end{proof}

\noindent\textbf{Conflicts of interest.} The author declares that there are
no conflicts of interest.

\medskip \noindent\textbf{Data availability.} Not applicable.

\medskip \noindent\textbf{Funding.} This research received no external funding.

\end{document}